\documentclass[12pt, a4paper]{amsart}

\usepackage{a4wide}
\usepackage[english]{babel}
\usepackage[T2A]{fontenc}
\usepackage[utf8]{inputenc} 
\usepackage{amsfonts}
\usepackage{amssymb, amsthm, amscd}
\usepackage{amsmath}
\usepackage{mathtools}
\usepackage{needspace}
\usepackage{etoolbox}
\usepackage{lipsum}
\usepackage{comment}
\usepackage{cmap}
\usepackage[pdftex]{graphicx}
\usepackage[unicode]{hyperref}
\usepackage[matrix,arrow,curve]{xy}
\usepackage[usenames,dvipsnames]{xcolor}
\usepackage{colortbl}
\usepackage{textcomp}
\usepackage{cite}
\usepackage{euscript}

\makeatletter
\def\@settitle{\begin{center}
		\baselineskip14\p@\relax
		\bfseries
		\LARGE
		\@title
	\end{center}
}
\makeatother

\usepackage{amsmath, 
	amssymb, 
	amsthm, 
	mathtools,indentfirst,hyperref,
	graphicx,xcolor,
	 textcomp,
	 mathrsfs
	}

\newtheorem{theorem}{Theorem}
\newtheorem{lemma}{Lemma}

\newtheorem{corollary}[theorem]{Corollary}
\theoremstyle{definition}
\newtheorem{definition}{Definition}
\newtheorem{remark}{Remark}

\renewcommand{\leq}{\leqslant}

\newcommand{\ot}{\otimes}
\newcommand{\x}{\otimes}

\DeclareMathOperator{\Ext}{Ext}
\DeclareMathOperator{\Hom}{Hom}

\title{$BV$-structure on Hochschild cohomology for exceptional local algebras of quaternion type. Case of even parameter.}
\author{Alexander Generalov and Andrei V. Semenov}

\thanks{This work is partially supported by Young Russian Mathematics award and the second author is grateful to its jury and sponsors; is supported by ``Native Towns'', a social investment program of PJSC ``Gazprom Neft''. Also the second author is supported in part by The Euler International Mathematical Institute, grant number is 075-15-2019-16-20. The first author is partially supported by the RFFI under Grant No. 20-01-00030.}

\address{Alexander Generalov:
Saint Petersburg State University, Universitetskaya nab., 7-9,
St. Petersburg, 199178 Russia}

\email{ageneralov@gmail.com}

\address{Andrei V. Semenov:
Chebyshev Laboratory, 
St. Petersburg State University, 14th Line V.O., 29B, 
Saint Petersburg 199178 Russia}

\email{asemenov.spb.56@gmail.com}


\begin{document}
\maketitle

\begin{abstract} We give a full description of the $BV$-structure on the Hochschild cohomology of exceptional local algebras of quaternion type, defined by parameters $(k,0,d)$ in case of even parameter $k \geqslant 3$, according to Erdmann's classification. We develop and use the method of comparison morphisms and weak self-homotopy method. This article states as a generalization of similar results about $BV$-structures on Hochscild cohomology of algebras of quaternion type.
\end{abstract}

\vspace{10mm}

\section{Introduction}
Algebras of dihedral, semi-dihedral and quaternion types arise naturally from Erdmann's classical book \cite{E} as a product of classification of tame blocks. The Hochschild cohomology of such algebras are well-studied by the large amount of mathematicians, such as A. I. Generalov (see \cite{G1}, \cite{G2}, \cite{GII}), A. A. Ivanov (see \cite{I1}, \cite{I2}), C. Cibils (see \cite{CS}) and many others: see also \cite{V2} and \cite{GA}, where one could find more references on studies of Hochschild cohomology. For an associative algebra $A$ there are many structures on its Hochschild cohomology algebra $HH^*(A)$ of $A$: for example, it has a graded commutative algebra structure (see \cite{H}) and graded Lie algebra structure, introduced by Gerstenhaber in his paper \cite{Ger}. T. Tradler was the first who described $BV$-structure on Hochschild cohomology of finite dimensional symmetric algebras (see \cite{Tr}). The problem here is that $BV$-structure is defined in terms of the bar-resolution. It is almost impossible to compute such structure for concrete examples, because the dimension of resolution's items grows exponentially. In order to avoid this problem we use a method of comparison morphisms (see also \cite{IIVZ} and \cite{I2}). \par 
The significance of $BV$-structure is that it gives a method to compute the Gerstenhaber Lie
bracket, which is important and hard-reached structure on $HH^*(A)$. In this paper we deal with the Hochschild cohomology algebra for algebras of quaternion type $R(k,0,d)$ over an algebraically closed field $K$ of characteristic 2, described in paper \cite{GA}. Some partial cases of this family of algebras were studied in \cite{IIVZ} for the case of $R(2,0,0)$ and in \cite{I2} for the case of $R(k,0,0)$. Here we study only case of even parameter $k \ge 3$, because in \cite{GA} it was shown that cases of even and odd parameter are differ significantly. \par 
Note that the calculation of these structures is a difficult task and there are only a few
examples of such calculations: for example, one should mention the papers of Menichi (see \cite{Me1, Me2}), Yang (see \cite{Y}), Tradler (see \cite{Tr}), Ivanov (see \cite{I2}) and two articles \cite{V1} and \cite{TGZ} about $BV$-structures of Frobenius algebras. 

\section{Main definitions and facts}
\subsection{Hochschild (co)homology}
For an associative algebra $A$ over a field $K$ it's $n$-th Hochschild cohomology is a vector space $HH^n(A) = \Ext^n_{A^e} (A,A)$ for $n \geqslant 0$, where $A^e= A \otimes A^{op}$ is the enveloping algebra for $A$. Notice that the {\it bar-resolution } is a free resolution 
$$\CD A @<{\mu = d_0}<< A^{\otimes 2} @<{d_1}<<  A^{\otimes 3} @<{d_2}<< \dots @<{d_n}<< A^{\otimes n+2} @<{d_{n+1}}<< A^{\otimes n+3} \dots \endCD$$
with differentials
$$d_n(a_0 \x \dots a_{n+1}) = \sum \limits_{i=0}^{n} (-1)^i a_0 \x \dots \x a_i a_{i+1} \x \dots  \x a_{n+1}.$$
One can construct the {\it normalized bar-resolution}, in which $n$-th element is given by formula $\overline{Bar}(A)_n = A \otimes \overline{A}^{\x n} \x A$, where $\overline{A} = A/ \langle 1_A \rangle$, and the differentials are induced by that of the bar-resolution. \par 

We define the $n$-th Hochschild homology space $HH_n(A)$ as follows:
$$HH_n(A) = H_n(A \x_{A^e} Bar_{\bullet}(A)) \simeq H_n(A^{\x (\bullet+1)}),$$
where the differentials $\partial_n : A^{\x (n+1)} \longrightarrow A^{\x n}$ comes by mapping $a_0 \x \dots \x a_n$ to
$$\sum_{i=0}^{n-1} (-1)^i a_0 \x \dots \x a_{i} a_{i+1} \x \dots \x a_n + (-1)^{n} a_na_0 \x \dots \x a_{n-1}.$$

Let's look closely on complex $(\Hom_{A^e}(Bar_{\bullet}(A), A), \ \delta^{\bullet})$. As always,
$$HH^{\bullet}(A) = H^{\bullet}(\Hom_{A^e}(Bar_{\bullet}(A),A)) \simeq H^{\bullet}(\Hom_k(A^{\x \bullet}, A)),$$
and for $f \in \Hom_k(A^{\x n}, A)$ the element $\delta^n (f)$ maps $a_1 \x \dots \x a_{n+1}$ to
$$a_1 f(a_2 \x \dots \x a_{n+1}) + \sum_1^n (-1)^i f(a_1 \x \dots \x a_i a_{i+1} \x \dots \x a_{n+1}) + (-1)^{n+1}f(a_1 \x \dots \x a_{n})a_{n+1}.$$
One can describe a cup-product on $HH^*(A)$: for classes $a \in HH^n(A)$ and $b \in HH^m(A)$ its cup-product $a \smile b \in HH^{n+m}(A)$ is defined by the class of the cup-product of representatives $a \in \Hom_k(A^{\x n}, A)$ and $b \in \Hom_k(A^{\x m}, A)$. So by linear extension
$$\smile : HH^n(A) \times HH^m(A) \longrightarrow HH^{n+m}(A)$$
the Hochschild cohomology space $HH^{\bullet}(A) = \bigoplus \limits_{n \geqslant 0} HH^n(A)$ becomes a graded-commutative algebra. 
\subsection{Gerstenhaber bracket}
For $f \in \Hom_k(A^{\x n}, A)$ and $g \in \Hom_k(A^{\x m}, A)$ one can define $f \circ_i g \in \Hom_k(A^{\x n+m-1}, A)$ by the following rules:
\begin{enumerate}
    \item if $n \geqslant 1$ and $m \geqslant 1$, put $f \circ_i g (a_1 \x \dots \x a_{n+m-1}) = f(a_1 \x \dots a_{i-1} \x g(a_i \x \dots a_{i+m-1}) \x \dots \x a_{n+m-1})$.
    \item if $n \geqslant 1$ and $m=0$, put $f \circ_i g (a_1 \x \dots \x a_{n-1}) = f(a_1 \x \dots a_{i-1} \x g \x a_i \x  \dots \x a_{n+m-1})$, because $g \in A$ in this case. 
    \item otherwise set $f \circ_i g = 0$.
\end{enumerate}
So put $a \circ b = \sum \limits_{i=1}^n (-1)^{(m-1)(i-1)} a \circ_i b$.
\begin{definition}
For any $f \in \Hom_k(A^{\x n}, A)$ and $g \in \Hom_k(A^{\x m}, A)$ we define the {\it Gerstenhaber bracket of f and g} by the formula 
$$[f,g] = f \circ g - (-1)^{(n-1)(m-1)} g \circ f.$$
\end{definition}
This bracket obviously lies in $\Hom_k(A^{\x n+m-1}, A)$, so for $a \in HH^n(A)$ and $b \in HH^m(A)$ we can define $[a,b] \in HH^{n+m-1}(A)$ as a class of Gerstenhaber bracket for representatives $a$ and $b$. This bracket correctly induces map
$$[-,-]:HH^{*} (A) \times HH^{*} (A) \longrightarrow HH^{*} (A),$$
which gives us a structure of graded Lie algebra on Hochschild cohomology. Also one can show that $(HH^{*} (A), \smile, [-,-])$ is a Gerstenhaber algebra (see \cite{Ger}).

\subsection{$BV$-structure}
\begin{definition}
The {\it Batalin-Vilkovisky algebra} (or $BV$-algebra) is a Gerstenhaber algebra $(A^{\bullet}, \smile, [-,-])$ together with an operator $\Delta^{\bullet}$ of degree $-1$, such that $\Delta \circ \Delta = 0$ and 
$$[a,b] = - (-1)^{(|a|-1)|b|} \big(\Delta(a \smile b) - \Delta(a) \smile b - (-1)^{|a|} a \smile \Delta(b) \big)$$
for homogeneous $a, b \in A^{\bullet}$.
\end{definition}
Note that in this definition, we use the signs as in \cite{Tr} (cf. \cite{JG}, Remark 2.5).\par 
For $a_0 \x \dots \x a_n \in A^{\x (n+1)}$ define
$$\mathfrak{B}(a_0 \x \dots \x a_n) = \sum_{i=0}^n (-1)^{in} 1 \x a_i \x \dots \x a_n \x a_0 \x \dots \x a_{i-1} + $$
$$+\sum_{i=0}^n (-1)^{in} a_i \x 1 \x a_{i+1} \x \dots \x a_n \x a_0 \x \dots \x a_{i-1}.$$
Obviously $\mathfrak{B}(a_0 \x \dots \x a_n) \in A^{\x (n+2)} \simeq A \x_{A^e} A^{\x (n+3)}$, hence it can be lifted to the chain map of complexes. Observe that $\mathfrak{B} \circ \mathfrak{B} = 0$, so one can correctly define induced map on $HH^*(A)$.
\begin{definition}
The above defined map $\mathfrak{B}: HH_{\bullet}(A) \longrightarrow HH_{\bullet+1}(A)$ is said to be Connes' $\mathfrak{B}$-operator.
\end{definition}
\begin{definition}
An algebra $A$ is said to be symmetric, if it is isomorphic (as $A^e$-module) to its dual $DA = \Hom_K(A,K)$.
\end{definition} 
For a symmetric algebra $A$ one can always find a non-degenerate symmetric biliniar form $\langle -,- \rangle : A \times A \longrightarrow K$. Obviously, the reversed statement holds: for any such form the corresponding algebra $A$ is symmetric. So in the case of symmetric algebras the Hochschild homology and colomology are dual to each other:
$$\Hom_K(A \x_{A^e} Bar_{\bullet}(A),K) \simeq \Hom_{A^e} (Bar_{\bullet}(A), \Hom_K(A,K)) \simeq \Hom_{A^e}(Bar_{\bullet}(A), A),$$
hence there exists an operator $\Delta: HH^n(A) \longrightarrow HH^{n-1}(A)$, which corresponds to Connes' $\mathfrak{B}$-operator. \par 
So for symmetric $A$ the algebra $HH^*(A)$ is clearly a $BV$-algebra (see \cite{Tr} for details), and, as noticed, Connes' operator for homology corresponds to $\Delta$ operator on cohomology.
\begin{theorem}[Theorem 1 in \cite{M}]
The above defined cup-product, Gerstenhaber bracket and operator $\Delta$ induce a structure of $BV$-algebra on  $HH^{*}(A)$. Moreover, for $f \in \Hom_K(A^{\x n}, A)$ the element $\Delta(f) \in \Hom_K(A^{\x (n-1)}, A)$ can be calculated by the formula
$$\langle \Delta(f)(a_1 \x \dots \x a_{n-1}),a_n \rangle = \sum_{i=1}^n (-1)^{i(n-1)} \langle f(a_i \x \dots \x a_{n-1}\x a_n \x a_1 \x \dots \x a_{i-1}), 1 \rangle$$
for any $a_i \in A$.
\end{theorem}

\begin{remark} All constructions here can be described in terms of the normalized bar-resolution.
\end{remark}

\section{Weak self-homotopy}
\subsection{The resolution}
Let $K$ be an algebraically closed field of characteristic 2 and let $c, d \in K$. Define $R(k, c, d) = K \langle X,Y \rangle /I$, where $I$ is an ideal in $K \langle X,Y \rangle$ spanned by $X^2+Y(XY)^{k-1} + c(XY)^k, Y^2+X(YX)^{k-1} + d(XY)^k, X(YX)^k, Y(XY)^k$. \par
It can be easily seen that $(XY)^k + (YX)^k \in I$. Now let $B$ be the standard basis of $R=R(k, c, d)$: we recall that $B = \{1, x(yx)^i, y(xy)^i, (xy)^i, (yx)^i\}_{i=0}^{k-1} \cup \{(xy)^k\}$ (see \cite{GA}). So the set $B_1 = \{u \otimes v \mid u, v \in B\}$ is the basis for the enveloping algebra $\Lambda = R \otimes R^{op}$.\par
Algebras of the family $R(k,c,d)$ are symmetric, since there exists a non-degenerate symmetric bilinear form
$$\langle b_1, b_2\rangle = \begin{cases}
 1, & b_1 b_2 \in Soc(R) \\
 0, & \text{otherwise.}
\end{cases}$$
In order to describe a structure of graded Lie algebra one should only know how $\Delta$ acts on $HH^*(A)$. In this article we are interested in the case $c=0$, so consider only $R(k,0,d)$: we often write $R$ for this algebra. \par
Observe that the right multiplication by $\lambda \in \Lambda$ induces an endomorphism $\lambda^{*}$ of the left $\Lambda$-module $\Lambda$. Also we will use an endomorphism of the right $\Lambda$-module $\Lambda$, induced by the left multiplication on $\lambda$, which we denote by ${}^{*}\lambda$. \par
For computations we construct 4-periodic complex in the category of (left) $\Lambda$-modules
$$
\CD
P_0 @<{d_0}<< P_1 @<{d_1}<<  P_2 @<{d_2}<< P_3  @<{d_3}<< P_4 @<{d_4}<< \dots\\
\endCD$$
where $P_0=P_3=\Lambda$, $P_1 = P_2 = \Lambda^2$ and differentials given by formulae
$$d_0 = \begin{pmatrix}
x \otimes 1 + 1\otimes x & y \otimes 1 + 1 \otimes y
\end{pmatrix}, 
\quad d_1=\begin{pmatrix}
d_{11}& d_{12} \\
d_{13} & d_{14}
\end{pmatrix},$$
$$d_2 = \begin{pmatrix}
x \otimes 1 + 1\otimes x \\
y \otimes 1 + 1\otimes y + dy\otimes y + 1 \otimes dx(yx)^{k-1}+d^2y\otimes x(yx)^{k-1}\\
\end{pmatrix}, \quad d_3 =  \lambda^{*},$$  
where
$$\begin{cases} 
d_{11} = x \otimes 1 + 1\otimes x + \sum \limits_{0}^{k-2} y(xy)^i \otimes y(xy)^{k-2-i}, \\
d_{12} = \sum \limits_{0}^{k-1}(xy)^i \otimes (yx)^{k-1-i} + d \sum \limits_{0}^{k-1} (xy)^i \otimes y(xy)^{k-1-i}, \\
d_{13} = \sum \limits_{0}^{k-1}(yx)^i \otimes (xy)^{k-1-i},\\
d_{14} = y \otimes 1 + 1\otimes y + \sum \limits_{0}^{k-2} x(yx)^i \otimes x(yx)^{k-2-i} + d \sum \limits_{0}^{k-1} x(yx)^i \otimes (xy)^{k-i-1},
\end{cases} $$
and
$$ \lambda = \sum \limits_{0}^{k}(xy)^i \otimes (xy)^{k-i} + \sum \limits_{1}^{k-1} (yx)^i \otimes (yx)^{k-i} +\sum \limits_{0}^{k-1}y(xy)^i \otimes x(yx)^{k-i-1} +$$ $$+\sum \limits_{0}^{k-1}x(yx)^i \otimes y(xy)^{k-i-1} + dx(yx)^{k-1} \otimes x(yx)^{k-1}.$$

Also consider the map $\mu:\Lambda \longrightarrow R$ induced by multiplication: $\mu(a \otimes b) = ab$.
\begin{theorem}[Proposition 3.1 in \cite{GA}]
The complex $P_{\bullet}$ equipped with the map $\mu$ forms the minimal $\Lambda$-projective resolution of $R$.
\end{theorem}

It is useful to rewrite the resolution as in \cite{IIVZ}: using the quiver $Q_1$ of $R$ one could define modules $KQ_1 = \langle x,y\rangle$, and $KQ_1^* = \langle r_x,r_y\rangle$, where $r_x = x^2+y(xy)^{k-1}$, $r_y = y^2 + x(yx)^{k-1}+d(xy)^k$. It is easy to see that
$$R \otimes KQ_1 \otimes R = R \otimes \langle x \rangle \otimes R \oplus R \otimes \langle y \rangle \otimes R \simeq R \otimes R^{op} \oplus R \otimes R^{op} = \Lambda \oplus \Lambda,$$ 
so we can represent the resolution $P_{\bullet}$ in the following form:
$$
\CD
R @<{\mu}<< R \otimes R @<{d_0}<< R\otimes KQ_1 \otimes R @<{d_1}<< R\otimes KQ_1^* \otimes R  @<{d_2}<< R\otimes R @<{d_3}<< \dots
\endCD$$
where $P_{n+4} = P_n$ for $n \in \mathbb{N}$. The differentials are given by formulae 
\begin{itemize}
    \item $d_0(1 \otimes x \otimes 1) = x \otimes 1 + 1 \otimes x$, $d_0(1 \otimes y \otimes 1) = y \otimes 1 + 1 \otimes y$;
    \item $d_1(1 \otimes r_x \otimes 1) = 1\otimes x \otimes x + x \otimes x \otimes 1 + \sum \limits_{i=0}^{k-2} y(xy)^i \otimes x \otimes y(xy)^{k-2-i} + \sum \limits_{i=0}^{k-1} (yx)^i \otimes y \otimes (xy)^{k-1-i}$, \\
    $d_1(1 \otimes r_y \otimes 1) = 1\otimes y \otimes y + y \otimes y \otimes 1 + \sum \limits_{i=0}^{k-2} x(yx)^i \otimes y \otimes x(yx)^{k-2-i} + d\sum \limits_{i=0}^{k-1} x(yx)^i \otimes y \otimes (xy)^{k-1-i} + \sum \limits_{i=0}^{k-1} (xy)^i \otimes x \otimes (yx)^{k-1-i} + d\sum \limits_{i=0}^{k-1} (xy)^i \otimes x \otimes y(xy)^{k-1-i}$;
    \item $d_2(1\otimes 1) = x \otimes r_x \otimes 1 + 1\otimes r_x \otimes x+ y\otimes r_y \otimes 1 + 1\otimes r_y \otimes y + d y\otimes r_y \otimes y + d \otimes r_y \otimes x(yx)^{k-1} + d^2 y \otimes r_y \otimes x(yx)^{k-1}$;
    \item $d_3 = \rho \mu$, where $\rho (1) = \sum \limits_{b \in B} b^*\otimes b + d x(yx)^{k-1} \otimes x(yx)^{k-1}$ and $\mu : R \otimes R \longrightarrow R$ is a multiplication map.
\end{itemize}
\subsection{Construction}
\begin{definition}
For the complex 
$$\CD
0 @<{}<< N @<{d_0}<< Q_0 @<{d_1}<< Q_1  @<{d_2}<< Q_2 \dots
\endCD$$
we define {\it the weak self-homotopy} as a collection of $K$-homomorphisms $t_{n+1} : Q_n \longrightarrow Q_{n+1} $ and $t_{0} : N \longrightarrow Q_0$ such that $t_{n}d_n + d_{n+1}t_{n+1} = id_{Q_n}$ for all $n \geqslant 0$ and $d_0t_{0} = id_N$.
\end{definition}

Now we need to construct a weak self-homotopy $\{t_i :P_i \longrightarrow P_{i+1} \}_{i \geqslant -1}$ for such projective resolution, as in \cite{BZZ} (here $P_{-1} = R$). 
In order to do this let us define bimodule derivation $C : KQ \longrightarrow KQ \otimes KQ_1 \otimes KQ$ by sending the path $\alpha_1...\alpha_n$ to $\sum \limits_{i=1}^n \alpha_1 ... \alpha_{i-1} \otimes \alpha_i \otimes \alpha_{i+1} \dots \alpha_n$, and consider induced map $C: R \longrightarrow R \otimes KQ_1 \otimes R$. 
So one could define $t_{-1} (1) = 1\otimes 1$ and $t_0 (b \otimes 1) = C(b)$ for $b \in B$. Now construct $t_1 : P_1 \longrightarrow P_2$ by the following rules: for $b\in B$ let
$$t_1(b \otimes x \ot 1) = $$

$$ = \begin{cases} 
0, & bx \in B\setminus\{y(xy)^{k-1}\}\\

1 \ot r_x \ot 1, & b=x\\

1 \x r_x \x x^2 + x \x r_x \x x + x^2 \x r_x \x 1 + \\ (yx)^{k-1} \x r_y \x (xy)^{k-1}, & b = (xy)^k\\

T\big(y \x r_x \x 1 + (xy)^{k-1} \x r_x \x y(xy)^{k-2} + 1 \x r_y \x (xy)^{k-1} \\ +dx \x r_x \x (xy)^{k-1}  + dy(xy)^{k-1}\x r_x \x y(xy)^{k-2}\big), & b = Tyx\\

y \x r_y \x 1+1 \x r_y \x y + dy \x r_y \x y \\+ d\x r_y \x x(yx)^{k-1} + d^2 y\x r_y \x x(yx)^{k-1}, & b = y(xy)^{k-1}
\end{cases}
$$
and let
$$t_1(b \otimes y \ot 1) = $$
$$ = \begin{cases} 
0, & by \in B \\

1 \ot r_y \ot 1, & b=y\\

T\big(x \x r_y \x 1 + (yx)^{k-1} \x r_y \x x(yx)^{k-2} + 1 \x r_x \x (yx)^{k-1}\\
+ d \x r_x \x y(xy)^{k-1} + d(yx)^{k-1} \x r_y \x (xy)^{k-1} \big), & b = Txy.\\
\end{cases}
$$

In order to define $t_2:P_2 \longrightarrow P_3$ one can put
\begin{itemize}
   \item $t_2(x \x r_x \x 1) = 1 \x 1$,
   \item $t_2 (xy \x r_x \x 1) =  d y(xy)^{k-1} \x y(xy)^{k-2}$,
   \item $t_2(y(xy)^{k-1} \x r_x \x 1) =  1\x x$,
   \item $t_2((xy)^k \x r_x \x 1) = \sum \limits_{i=0}^{k-1} ((yx)^i \x y(xy)^{k-i-1} + y(xy)^i \x (xy)^{k-i-1}) +dx(yx)^{k-1} \x (xy)^{k-1}$,
   \item $t_2((yx)^i \x r_x \x 1) =  \sum \limits_{j=1}^{i-1} (yx)^j \x y(xy)^{i-j-1} +\sum \limits_{j=1}^{i} y(xy)^{j-1} \x (xy)^{i-j} + dx(yx)^{k-1} \x (xy)^{i-1}$ for $i>1$,
   \item $t_2(yx \x r_x \x 1) = y \x 1 + dx(yx)^{k-1} \x 1  + d(xy)^{k-1} \x x + \sum \limits_{i=0}^{k-2} dx(yx)^i \x (yx)^{k-i-1} + \sum \limits_{i=1}^{k-2} d(xy)^i \x x(yx)^{k-i-1}$,
   \item $t_2 (x(yx)^i \x r_x \x 1) =  \sum \limits_{j=1}^{i}( (xy)^j \x (xy)^{i-j} + x(yx)^{j-1} \x y(xy)^{i-j})  + \delta_{i,1} dy(xy)^{k-1} \x (yx)^{k-1},$
   \item $t_2(b \x r_x \x 1)=0$ otherwise;
\end{itemize}
and let
 \begin{itemize}
    \item $t_2((xy)^i \x r_y \x 1) =  \sum \limits_{j=1}^{i-1} (xy)^j \x x(yx)^{i-j-1} +\sum \limits_{j=1}^{i} x(yx)^{j-1} \x (yx)^{i-j} + d\sum \limits_{j=1}^{i} x(yx)^{j-1} \x y(xy)^{i-j} + d\sum \limits_{j=1}^{i-1} (xy)^j \x (xy)^{i-j}$ for $i>1$,
    \item $t_2(xy \x r_y \x 1) = x \x 1 + dx \x y$,
    \item $t_2(y(xy)^i \x r_y \x 1) = \sum \limits_{j=1}^i ((yx)^j \x (yx)^{i-j} + y(xy)^{j-1} \x x(yx)^{i-j} ) + d \sum \limits_{j=1}^i ((yx)^j \x y(xy)^{i-j} + y(xy)^{j-1} \x (xy)^{i-j+1} ) +d^2 x(yx)^{k-1}\x (xy)^i + dx(yx)^{k-1} \x x(yx)^{i-1}+ \delta_{i,1} \cdot d (xy)^{k-1}\x y(xy)^{k-1}$ for $1 \leq i \leq k-1$,
    \item $t_2(b \x r_y \x 1) = 0$ otherwise.
\end{itemize} 
Finally, let $t_3 : R \x R \longrightarrow R \x R$ be the map defined by rules $t_3((xy)^k \x 1) = 1 \x 1$ and $t_3(b \x 1)=0$ elsewhere. It remains to put $t_{n+4} = t_n$ for any $n \geqslant 4$.\par 
\begin{theorem}
The above-defined family of maps $\{t_i : P_{i} \longrightarrow P_{i+1}\}_{i=0}^{+\infty}$ together with $t_{-1}: R \longrightarrow P_0$ forms a weak self-homotopy for the resolution $P_{\bullet}$.
\end{theorem}
\begin{proof}
For any $n \in \mathbb{N}$ it remains to verify a commutativity of required diagrams, which is straight-up obvious from definitions of $t_n$ for $n \leq 4$ and from periodicity for $n \geqslant 5$.
\end{proof}

\section{Comparison morphisms}
Consider the normalized bar-resolution $ \overline{{Bar}}_{\bullet} (R) = R \otimes  \overline{{R}}^{\x \bullet} \otimes R$, where $ \overline{R} = R / (k\cdot 1_R)$. We now need to construct a comparison morphisms between $P_{\bullet}$ and $\overline{Bar_{\bullet}} (R)$
$$\Phi : P_{\bullet} \longrightarrow  \overline{{Bar}}_{\bullet}(R) \text{ and } \Psi :  \overline{{Bar}}_{\bullet}(R) \longrightarrow  P_{\bullet}.$$
It is easy to check that there exists a weak self-homotopy defined by the formula $s_n(a_0 \x ... \x a_n \x 1) = 1 \x a_0 \x ... \x a_n \x 1$ over $ \overline{{Bar}}_{\bullet} (R)$, so one can put $\Phi_n = s_{n-1} \Phi_{n-1} d^P_{n-1}$ and $\Phi_0 = id_{R \x R}$.
\begin{lemma}
Let $\Psi :  \overline{{Bar}}_{\bullet}(R) \longrightarrow  P_{\bullet}$ be the chain map constructed using $t_{\bullet}$. Then for any $n \in \mathbb{N}$ and any $a_i \in R$ the following formula holds:
$$\Psi_n (1 \x a_1 \x ... \x a_n \x 1) = t_{n-1} (a_1 \Psi_{n-1} (1 \x a_2 \x ... \x a_n \x 1)).$$
\end{lemma}
\begin{proof}
It follows from Lemma 2.5 of \cite{IIVZ}.
\end{proof}
Let us, for example, directly compute first items of $\Phi_{\bullet}$:
\begin{enumerate}
    \item The map $\Phi_1$ is induced by embedding $R \x kQ_1 \x R \longrightarrow R \otimes  \overline{R} \otimes R$,
    \item $\Phi_2 (1 \x r_x \x 1) = 1 \x x \x x \x 1+ \sum \limits_{i=0}^{k-2} 1 \x y(xy)^i \x x \x y(xy)^{k-i-2} + \sum \limits_{i=1}^{k-1} 1 \x (yx)^i \x y \x (xy)^{k-i-1}$,\\
    
$\Phi_2( 1 \x r_y \x 1) =  1 \x y \x y \x 1+ \sum \limits_{i=0}^{k-2} 1 \x x(yx)^i \x y \x x(yx)^{k-i-2} + d \sum \limits_{i=0}^{k-1} 1 \x x(yx)^i \x y \x (xy)^{k-i-1} + \sum \limits_{i=1}^{k-1} 1 \x (xy)^i \x x \x (yx)^{k-i-1} + d \sum \limits_{i=1}^{k-1} 1 \x (xy)^i \x x \x y(xy)^{k-i-1}$,

\item $\Phi_3(1\x 1) = 1 \x x \x x \x x \x 1 + \sum \limits_{i=0}^{k-2} 1 \x x \x y(xy)^i \x x \x y(xy)^{k-i-2} + \sum \limits_{i=1}^{k-1} 1 \x x \x (yx)^i \x y \x (xy)^{k-i-1} + 1 \x y \x y \x y \x 1+ d \x y \x x(yx)^{k-1} \x y \x 1+ \sum \limits_{i=0}^{k-2} 1 \x y \x x(yx)^i \x y \x x(yx)^{k-i-2} + \sum \limits_{i=1}^{k-1} 1 \x y \x (xy)^i \x x \x (yx)^{k-i-1} + d \x y \x y \x y \x y +d^2 \x y \x x(yx)^{k-1} \x y \x y + d^2 \x y \x y \x y \x x(yx)^{k-1} + d^3 \x y \x x(yx)^{k-1} \x y \x x(yx)^{k-1}$,

\item $\Phi_4(1 \x 1) = \sum_{b \in B} 1 \x b \Phi_3(1 \x 1) b^* + d \x x(yx)^{k-1} \Phi_3(1\x 1) x(yx)^{k-1}$.
\end{enumerate}
Also it is easy to see that $\Psi_0 = id_{R \x R}$, $\Psi_1(1 \x b \x 1) = C(b)$, $\Psi_2(1 \x a_1 \x a_2 \x 1) = t_1 \big(a_1 C(a_2) \big)$, and so on: we only interested in recursive structure like in Lemma 1.\par 
In order to obtain $BV$-structure on the Hochschild cohomology one needs to compute $\Delta : HH^n(R) \longrightarrow HH^{n-1}(R)$. By the Poisson rule 
$$[a \smile b, c] = [a,c] \smile b + (-1)^{|a| (|c|-1)}(a \smile [b,c]),$$
so, because char$K = 2$ we obtain
$$\Delta(abc) = \Delta(ab)c + \Delta(ac)b + \Delta(bc)a + \Delta(a)bc +\Delta(b)ac + \Delta(c)ab,$$
hence we only need to know $\Delta$ on generating elements of $HH^*(R)$ and also on the cup-products of such elements. Also for $\alpha \in HH^n(R)$ there exists cocycle $f \in \Hom (P_n,R)$ such that the following formula holds: $\Delta(\alpha) = \Delta(f \Psi_n) \Phi_{n-1}$. So
$$\Delta(\alpha) (a_1 \x ... \x a_{n-1}) = \sum_{b \in B\setminus \{1\}} \langle \sum \limits_{i=1}^n (-1)^{i(n-1)} \alpha(a_i \x ... \x a_{n-1} \x b \x a_1 \x ...\x a_{i-1}), 1\rangle b^*,$$
where $\langle b,c \rangle$ is the bilinear form defined above.

\section{$BV$-structure}
For an algebraically closed field $K$ of characteristic 2 consider 
$$\mathcal{X} = \{p_1, p_2,p_3,p_4,q_1,q_2,w_1,w_2,w_3,e\},$$ where 
$$
 |p_1| = |p_2| = |p_3| = |p_4| = 0, \ |q_1| = |q_2| = 1, \ |w_1| = |w_2| =|w_3|=2, \ |e|= 4
$$
and ideal $\mathcal{I}$ in $K[\mathcal{X}]$ spanned by elements
\begin{itemize}
    \item of degree 0: $p_1^k, p_2^2, p_3^2, p_4^2$ and $p_i p_j$ for $i \not=j$;
    \item of degree 1: $p_3q_1+p_2q_2$, $p_1^{k-1}q_1+dp_3q_1$, $p_1q_2+p_2q_1, p_1^2q_2$;
    \item of degree 2: $q_1q_2$, $p_1^{k-1}w_3+dp_2w_2$, $p_2w_1, p_4w_1, p_3w_2, p_4w_2, p_4w_3$, $p_2q_1^2, p_3q_2^2$, $p_1w_1+p_2w_2, p_1w_1+p_3w_3, p_1w_1 + p_4q_1^2$, $p_3w_1+p_1w_2, p_3w_1+p_2w_3, p_3w_1+p_4q_2^2$
    \item of degree 3: $q_1w_1+q_2w_2$, $q_1^3+q_2^3+{{d^3(k+2)} \over {2}}p_1q_1w_1$, $p_3q_2w_1+p_1q_2w_2$, $p_3q_2w_1+p_2q_2w_3$, $p_1^{k-2}q_1w_3+dq_2w_2$, $p_1^{k-2} q_2w_3$, $p_1q_2w_1$, $p_1q_2w_3$, $q_1w_2+q_2w_3$;
    \item of degree 4: $w_3^2+p_1^2e,$ $q_2^2w_1$, $q_1^2w_3$, $q_2^2w_1$, $q_2^2w_2$, $w_1^2$, $w_2^2$, $w_iw_j$ for $i \not=j$.
    \end{itemize}
\begin{theorem}[Theorem 2.1 in \cite{GA}]
$HH^*(R) \simeq \mathcal{A} = K[\mathcal{X}]/\mathcal{I}$.
\end{theorem}
For computations one needs to know a simple form of these elements. Let $P$ be an item of minimal projective resolution  $R$. If $P= R \x R$, then denote by $f$ the homomorphism in  $\Hom_{R^e}(P,R)$ which sends $1 \x 1$ to $f$. If $P= R \x KQ\x R$ (or $P= R \x KQ_1\x R$), then denote by $(f,g)$ the homomorphism which sends $1\x x \x 1$ (or $1 \x r_x \x 1$) to $f$ and $1 \x y \x 1$ (or $1 \x r_y \x 1$) to $g$. So one can rewrite the generating elements like in \cite{GA}.
$$\begin{cases}
\text{elements of degree 0:} & p_1 = xy+yx, \ p_2 = x(yx)^{k-1}, \ p_3 = y(xy)^{k-1}, \ p_4 = (xy)^k, \\
\text{elements of degree 1:} & q_1 = ( y(xy)^{k-2}, \ 1+dy ),\ q_2 =(1, \  d(xy)^{k-1} +x(yx)^{k-2}), \\
\text{elements of degree 2:} & w_1 = ( x, \  0),\ w_2 = (0, \ y ),\ w_3= ( y,\ x + dxy ),\\
\text{elements of degree 4:} & e=1.\end{cases}$$
\begin{remark}
Also note that
$$C(b) = \begin{cases} \sum \limits_{j=0}^{i-1} (xy)^j \x x \x y(xy)^{i-j-1} + \sum \limits_{j=0}^{i-1} x(yx)^j \x y \x (xy)^{i-j-1}, & b = (xy)^i \\
 \sum \limits_{j=0}^{i} (xy)^j \x x \x (yx)^{i-j} + \sum \limits_{j=0}^{i-1} x(yx)^j \x y \x x(yx)^{i-j-1}, & b = x(yx)^i \\ 
  \sum \limits_{j=0}^{i} (yx)^j \x y \x (xy)^{i-j} + \sum \limits_{j=0}^{i-1} y(xy)^j \x x \x y(xy)^{i-j-1}, & b = y(xy)^i\\
   \sum \limits_{j=0}^{i-1} (yx)^j \x y \x x(yx)^{i-j-1} + \sum \limits_{j=0}^{i-1} y(xy)^j \x x \x (yx)^{i-j-1}, & b = (yx)^i.
 
 \end{cases}$$
\end{remark}

\subsection{Low degree cases}
Obviously $\Delta$ is equal to zero on each combination of elements of degree zero, because it is a morphism of degree $-1$.
\begin{lemma} For elements of degree 1 in $HH^*(R)$ the following statements hold:
$\Delta(q_1) = \Delta(q_2)= \Delta(p_3q_2) = 0$, $\Delta(p_1q_1) = dp_1$, $\Delta(p_1q_2) = \Delta(p_2q_1)= dp_2$, $\Delta(p_4q_1) = p_2$, $\Delta(p_3q_1) = \Delta(p_2q_2)= p_1^{k-1}$, $\Delta(p_4q_2) = p_3$.
\end{lemma}
\begin{proof}
We have already seen that $\Delta(a)(1\x 1) = \sum \limits_{b \not=1} \langle a (C(b)),1\rangle b^*$, so we only need to compute $\langle a (C(b)),1\rangle$ on elements of degree 1. It is easy to check that
$$p_1q_1 = (y(xy)^{k-1}, \ dyxy ), \quad p_2q_1 = ( 0, \ x(yx)^{k-1}+d(xy)^k ),$$ 
$$p_3q_1 = ( 0, \ y(xy)^{k-1}), \quad p_4q_1 =( 0, \ (xy)^k ),$$
and also one can do the same for each $p_iq_2$ by symmetry. Now it is clear that
$$\langle a (C(b)),1\rangle =  \begin{cases} dk, & a = q_1, b = (xy)^k \\ 
d(k-1), & a=p_1q_1, b \in \{(xy)^{k-1}, (yx)^{k-1} \}\\
d, & a \in \{ p_1q_2, p_2q_1 \}, b=y\\
1, & a \in \{p_3q_1, p_2q_2 \}, b \in \{xy, yx\} \text{ or } a=p_4q_1, b = y  \text{ or } a=p_4q_2, b = x \\
0, & \text{otherwise,}
\end{cases}$$
hence required formulae hold.
\end{proof}

\begin{lemma}
$\Delta(x)=0$ for any monomials $x \in HH^2(R)$.
\end{lemma}
\begin{proof}
If $a \in HH^2(R)$, then
$$\Delta(a)(1 \x x \x 1) = \Delta(a\Psi_2)\Phi_1(1\x x \x 1) = \sum_{b\not=1} \langle (a\Psi_2) (b \x x+x\x b), 1 \rangle b^*,$$
$$\Delta(a)(1 \x y \x 1) =\Delta(a\Psi_2)\Phi_1(1\x y \x 1) = \sum_{b\not=1} \langle (a\Psi_2) (b \x y+y\x b), 1 \rangle b^*.$$
Observe that $\Psi_2(1 \x b \x x \x 1+1\x x \x b \x 1) = t_1(b \x x \x 1 + xC(b))$. Obviously, $\Delta(q_1q_2)=0$. Furthermore, we have 
$$q_2^2 = ( 1, \ 0 ), \quad q_1^2 = ( 0, \ 1+d^2y^2+ {{d^3k} \over {2}}\cdot (xy)^k ).$$
One needs to compute $\Psi_2(1\x b \x x \x 1 + 1\x x \x b \x 1)$:
\begin{enumerate}
\item if $b=(xy)^i$ for $1 \leq i \leq k-1$, then $t_1 \big( b \x x \x 1 + xC(b) \big) = 1 \x r_x \x y(xy)^{i-1} + (yx)^{k-1} \x r_y \x (xy)^{i-1}+  \delta_{i,1} \big( y(xy)^{k-2} \x r_x \x (yx)^{k-1} + dy(xy)^{k-2} \x r_x \x y(xy)^{k-1} \big)$,
\item if $b = (yx)^i$ for $1 \leq i \leq k-1$, then $t_1 \big( b \x x \x 1 + xC(b) \big)  = y(xy)^{i-1} \x r_x \x 1+(yx)^{i-1} \x r_y \x (xy)^{k-1} + \delta_{i,1} \big( (xy)^{k-1} \x r_x \x y(xy)^{k-2} + dx \x r_x \x (xy)^{k-1} + dy(xy)^{k-1} \x r_x \x y(xy)^{k-2} \big)$,
\item if $b=x(yx)^i$ for $1 \leq i \leq k-1$, then $ t_1 \big( b \x x \x 1 + xC(b) \big)  = (xy)^i \x r_x \x 1+ x(yx)^{i-1} \x r_y \x (xy)^{k-1} + 1\x r_x \x (yx)^i  + (yx)^{k-1} \x r_y \x x(yx)^{i-1} + \delta_{i,1} \big( dx^2 \x r_x \x (xy)^{k-1} + d(xy)^k \x r_x \x y(xy)^{k-2} + dy(xy)^{k-2} \x r_x \x (xy)^k \big)$,
\item if $b=y(xy)^{k-1}$, then $ t_1 \big( b \x x \x 1 + xC(b) \big)  = y \x r_y \x 1 +1 \x r_y \x y + dy\x r_y \x y + d\x r_y \x x(yx)^{k-1} + d^2 y\x r_y \x x(yx)^{k-1}$,
\item if $b=(xy)^k$, then $ t_1 \big( b \x x \x 1 + xC(b) \big)  = x \x r_x \x x + x^2 \x r_x \x 1$,
\item if $b=x$ or $b=y(xy)^{i}$ for $0 \leq i \leq k-2$, then $t_1 \big( b \x x \x 1 + xC(b) \big)  =0$.
\end{enumerate}
Now consider $\Psi_2(1\x b \x y \x 1 + 1\x y \x b \x 1) = t_1 \big( b \x y \x 1 + yC(b) \big)$:
\begin{enumerate}
    \item if $b=(yx)^i$ for $1 \leq i \leq k-1$, then $t_1 \big( b \x y \x 1 + yC(b) \big) = 1 \x r_y \x x(yx)^{i-1} + (xy)^{k-1} \x r_x \x (yx)^{i-1} + d x\x r_x \x x(yx)^{i-1} + dx^2 \x r_x \x (yx)^{i-1}  + \delta_{i,1} \big( x(yx)^{k-2} \x r_y \x (xy)^{k-1} + d \x r_x \x x^2 + d (yx)^{k-1} \x r_y \x (xy)^{k-1} \big)$, 
   \item if $b = (xy)^i$ for $1 \leq i \leq k-1$, then $t_1 \big( b \x y \x 1 + yC(b) \big) = x(yx)^{i-1} \x r_y \x 1 +(xy)^{i-1} \x r_x \x (yx)^{k-1} + d (xy)^{i-1} \x r_x \x y(xy)^{k-1} + \delta_{i,1} \big( (yx)^{k-1} \x r_y \x x(yx)^{k-2} + d(yx)^{k-1} \x r_y \x (xy)^{k-1} \big)$,
   \item if $b = x(yx)^{k-1}$ then $t_1 \big( b \x y \x 1 + yC(b) \big) = y \x r_y\x 1 + 1 \x r_y \x y + d y \x r_y \x y + d \x r_y \x x(yx)^{k-1} + d^2 y \x r_y \x x(yx)^{k-1}$,
   \item if $b = y(xy)^i$ for $1 \leq i \leq k-1$, then $t_1 \big( b \x y \x 1 + yC(b) \big)  = (yx)^i \x r_y \x 1 + y(xy)^{i-1} \x r_x \x (yx)^{k-1} + dy(xy)^{i-1} \x r_x \x y(xy)^{k-1} + 1 \x r_y \x (xy)^i + (xy)^{k-1} \x r_x \x y(xy)^{i-1} + dx \x r_x \x (xy)^i + dx^2\x r_x \x y(xy)^{i-1}$,
   \item if $b = (xy)^k$ then $t_1 \big( b \x y \x 1 + yC(b) \big) = y \x r_y \x y + 1 \x r_y \x x(yx)^{k-1} + dy \x r_y \x x(yx)^{k-1}$,
   \item  $t_1 \big( b \x y \x 1 + yC(b) \big) = 0$ otherwise.
\end{enumerate}
Hence lemma holds by given computations.
\end{proof}

\subsection{Middle degree cases}
\begin{lemma}
We have $\Delta(q_1w_1) = \Delta(q_2w_2) = p_1^{k-2} w_3$, $\Delta(q_1w_2) = \Delta(q_2w_3) = \big( 1+d^3(l+1)p_4 \big) q_1^2 + dw_2$, $\Delta(q_1w_3) = q_2^2 + dw_3$, $\Delta(q_2w_1) = q_2^2$, where $l = {{k} \over {2}}$.
\end{lemma}
\begin{proof}
In this proof we use delta-like function
$$\mu_{a,b} = \begin{cases}
         1, & a \geqslant b, \\
         0, & a < b.
        \end{cases}$$
        Fix $a \in HH^3(R)$. By identifications $1 \x a_1 \x ... \x a_n \x 1 = a_1 \x ... \x a_n$ in $R$-bimodule $R \x \overline{R}^{\x n} \x R$ one can observe that
$$\Delta(a) (1 \x r_x \x 1) = \Delta(a \Psi_3)\Phi_2(1\x r_x \x 1) =  \sum \limits_{b \not= 1} \langle (a\Psi_3) (b \x x \x x + x \x b \x x + x \x x \x b), 1 \rangle b^* +$$ 
$$+\sum \limits_{b \not= 1} \sum \limits_{i=0}^{k-2} \langle (a\Psi_3) (b \x y(xy)^i \x x + y(xy)^i \x x \x b + x \x b \x y(xy)^i), 1 \rangle b^* \cdot y(xy)^{k-2-i} + 
$$
$$+\sum \limits_{b \not= 1} \sum \limits_{i=1}^{k-1} \langle (a\Psi_3) ((yx)^i \x y \x b + y \x b \x (yx)^i + b \x (yx)^i \x y), 1 \rangle b^* \cdot (xy)^{k-1-i}.
$$

It is easy to see that 
$$\Psi_3(b \x x \x x + x \x b \x x + x \x x \x b) = t_2 \Big( b \x r_x \x 1 +  xt_1 \big( b \x x \x 1 + xC(b) \big) \Big).$$
Denote this formula by $\Psi_3(b,x)$. 

\begin{itemize}
\item If $b=x$, then $\Psi_3(b,x) = 1 \x 1$,
\item if $b = x(yx)^i$ for $1 \leq i \leq k-1 $, then $\Psi_3(b,x) = \sum \limits_{j=1}^{i} \big( (xy)^j \x (xy)^{i-j} + x(yx)^{j-1} \x y(xy)^{i-j} \big) + 1 \x (yx)^i +  \delta_{i,1} \big( (yx)^{k-1} \x (xy)^{k-1}  + dy(xy)^{k-1} \x (yx)^{k-1} +dy(xy)^{k-1} \x (xy)^{k-1} \big)$,
\item if $b = y(xy)^{k-1}$, then $\Psi_3(b,x) = 1 \x x  + x \x 1 $,
\item if $b = (xy)^i$ for $1 \leq i \leq k-1$, then $\Psi_3(b,x)= 1\x y(xy)^{i-1} + \delta_{i,1} dy(xy)^{k-1} \x y(xy)^{k-2}$,
\item if $b = (yx)^i$ for $2 \leq i \leq k-1$, then $\Psi_3(b,x)= \sum \limits_{j=0}^{i-1} (yx)^j \x y(xy)^{i-j-1} + \sum \limits_{j=1}^{i} y(xy)^{j-1} \x (xy)^{i-j}+ dx(yx)^{k-1} \x (xy)^{i-1}$,
\item if $b = yx$, then $\Psi_3(b,x)= y \x 1 + d x(yx)^{k-1} \x 1 + d(xy)^{k-1} \x x + \sum \limits_{j=0}^{k-2} d x(yx)^j \x (yx)^{k-j-1} + \sum \limits_{j=1}^{k-2} d (xy)^j \x x(yx)^{k-j-1}$,
\item if $b=(xy)^k$, then $\Psi_3(b,x)= 1\x y(xy)^{k-1}$,
\item $\Psi_3(b,x) = 0$ otherwise.
\end{itemize}

Next for $1 \leq i \leq k-2$ one needs to compute
$$\Psi_3 (b \x y(xy)^i \x x + y(xy)^i \x x \x b + x \x b \x y(xy)^i) = $$
$$=t_2 \Big( b t_1 (y(xy)^i \x x \x 1) + y(xy)^i t_1\big( x C(b) \big) + x t_1 \big(b C(y(xy)^i) \big) \Big).$$ 
Denote this formula by $\Psi_3(2, b, x, i)$.
\begin{itemize}
\item If $b = x(yx)^j$ for $1 \leq j \leq k-1$, $i>0$ and $ i+j \ge k$, then $\Psi_3(2, b, x, i) = y(xy)^{k-1} \x y(xy)^{i+j-k} + \delta_{i+j,k} (yx)^{k-1} \x y^2$,
\item if $b = y$ and $i>0$, then $\Psi_3(2, b, x, i) = dy(xy)^{k-1} \x y(xy)^{i-1} + \delta_{i,1} d(yx)^{k-1} \x y^2 $,
\item if $b = xy$, then $\Psi_3(2, b, x, i) =  d\sum \limits_{l=1}^{k-1} \big( (yx)^l \x y(xy)^{k-1-l+i} + y(xy)^{l-1}\x (xy)^{k-l+i} \big) + d^2 x(yx)^{k-1} \x (xy)^{k-1+i} +
  \delta_{i,0} \big( 1 \x (yx)^{k-1} + dx(yx)^{k-1} \x x(yx)^{k-2}  + d \x y(xy)^{k-1} \big)+$
    $$+ \begin{cases} (yx)^{k-1} \x (xy)^i, & i>0\\ 
  \sum \limits_{l=1}^{k-1} (yx)^l \x (yx)^{k-1-l} +  \sum \limits_{l=1}^{k-1} y(xy)^{l-1} \x x(yx)^{k-1-l}, & i=0
  \end{cases}$$
 \item if $b = (yx)^j$ for $1 \leq j \leq k-1$, $i+j \ge k$ and $i>0$, then $\Psi_3(2, b, x, i) =  x \x y(xy)^{i+j-k}$,
 \item $\Psi_3(2, b, x, i) = 0$ otherwise.
 \end{itemize}
Only one step remains now: we need to describe
$$\Psi_3(3,b,x,i) := \Psi_3 \big( (yx)^i \x y \x b + b \x (yx)^i \x y + y \x b \x (yx)^i \big)$$
for $1 \leq i \leq k-1$.
\begin{itemize}
    \item If $b=x(yx)^j$ for $0 \leq j \leq k-1$, this formula can be rewritten as:
    \begin{enumerate}
        \item  $j=k-1$: $\sum \limits_{l=1}^i \big( (yx)^l \x (yx)^{i-l} + y(xy)^{l-1} \x x(yx)^{i-l} \big) + y \x x(yx)^{i-1} + \delta_{i,1} \big( y \x x^2 + d(xy)^{k-1} \x x^2 \big),$
        \item $i+j \geqslant k$, $j \not= k-1$: $y \x x(yx)^{i+j-k} + dx(yx)^{k-1} \x x(yx)^{i+j-k} + \delta_{i+j,k} d(xy)^{k-1} \x x^2,$
    \end{enumerate}
   \item if $b=(yx)^j$ for $1 \leq j \leq k$, this formula can be rewritten as:
      \begin{enumerate}
        \item  if $j=1$:   $dy(xy)^{i-1} \x x^2 + \delta_{i,1}  dx(yx)^{k-1} \x (xy)^{k-1} + \delta_{i,k-1} \big( d \sum \limits_{l=1}^{k-1} (xy)^{l} \x x(yx)^{k-l-1} + d \sum \limits_{l=1}^k x(yx)^{l-1} \x (yx)^{k-l} \big) $,
        \item if $1 < j \leq k-1$: $\mu_{i+j, k+1} \cdot  x(yx)^{k-1} \x x(yx)^{i+j-k-1}  + \delta_{i+j,k} \cdot \big(d \sum \limits_{l=1}^k x(yx)^{l-1} \x (yx)^{i+j-l} +d \sum \limits_{l=1}^{k-1} (xy)^{l} \x x(yx)^{i+j-l-1}  \big) + \delta_{i+j,k+1} \cdot  (xy)^{k-1} \x x^2$,
\item if $j=k$: $\delta_{i,1} (xy)^{k-1} \x x^2 + x(yx)^{k-1} \x x(yx)^{i-1}+ \sum \limits_{l=1}^i \big( (yx)^l \x y(xy)^{i-l} + y(xy)^{l-1} \x (xy)^{i-l+1} \big) + d x(yx)^{k-1} \x (xy)^i + d \sum \limits_{l=1}^k x(yx)^{l-1} \x (yx)^{k-l+i} + d \sum \limits_{l=1}^{k-1} (xy)^l \x x(yx)^{k+i-l-1}$.
    \end{enumerate}
    \item  if $b=y(xy)^j$ for $0 \leq j \leq k-1$, this formula can be rewritten as:
     \begin{enumerate}
        \item  $j=0$: $1 \x x(yx)^{i-1} + dy \x x(yx)^{i-1} + d^2 x(yx)^{k-1} \x x(yx)^{i-1} + \delta_{i,1}d^2 (xy)^{k-1} \x x^2$,
        \item $j=1$ and $i=1$: $d(xy)^{k-1} \x (xy)^k + d^2 y(xy)^{k-1} \x (xy)^k$,
    \end{enumerate}
    \item $\Psi_3(3, b, x, i) = 0$ otherwise.
\end{itemize}

Now one should deal with $1 \x r_y \x 1$:
$$\Delta(a) (1 \x r_y\x 1) = \sum \limits_{b \not=1} \langle \Delta(a) (\Psi_3 (b \x y \x y + y \x b \x y + y \x y \x b)),1\rangle b^* + $$
$$\sum \limits_{b \not=1} \sum\limits_{i=0}^{k-2} \langle \Delta(a) (\Psi_3 (b \x x(yx)^i \x y + x(yx)^i \x y \x b + y \x b \x x(yx)^i)),1\rangle b^* x(yx)^{k-2-i}+$$
$$ +d\sum \limits_{b \not=1} \sum\limits_{i=0}^{k-1} \langle \Delta(a) (\Psi_3 (b \x x(yx)^i \x y + x(yx)^i \x y \x b + y \x b \x x(yx)^i)),1\rangle b^* (xy)^{k-1-i}$$
$$ + \sum \limits_{b \not=1} \sum\limits_{i=1}^{k-1} \langle \Delta(a) (\Psi_3 ((xy)^i \x x\x b + x \x b \x (xy)^i +b \x (xy)^i \x x)),1\rangle b^* ((yx)^{k-i-1} + dy(xy)^{k-i-1})$$

Firstly, observe that
$$\Psi_3 (b \x y \x y + y \x b \x y + y \x y \x b) = t_2 \Big( b \x r_y \x 1 + yt_1(b \x y \x 1) + yt_1 \big( yC(b) \big) \Big).$$
Denote this formula by $\Psi_3(b,y)$.
\begin{itemize}
    \item If $b= (xy)^i$ for $1 \leq i \leq k-1$, then \begin{enumerate}
        \item if $i=1$ then $\Psi_3(b,y) = x \x 1 + d x \x y $,
        \item if $i>1$ then $\Psi_3(b,y) = \sum\limits_{l=1}^{i-1} (xy)^l \x x(yx)^{i-l-1} + \sum\limits_{l=1}^{i} x(yx)^{l-1} \x (yx)^{i-l}  + d\sum\limits_{l=1}^{i} x(yx)^{l-1} \x y(xy)^{i-l} + d\sum\limits_{l=1}^{i-1} (xy)^l \x (xy)^{i-l}$,
    \end{enumerate}
    \item if $b= (yx)^i$ for $1 \leq i \leq k-1$, then \begin{enumerate}
        \item if $i=1$ then $\Psi_3(b,y) = 1\x x  + dy \x x +  d^2 (xy)^{k-1} \x x^2 + d^2 x(yx)^{k-1} \x x $,
        \item if $i>1$ then $\Psi_3(b,y) = 1 \x x(yx)^{i-1} +  dy \x x(yx)^{i-1} + d^2 x(yx)^{k-1} \x x(yx)^{i-1}$.
    \end{enumerate}
    \item if $b= (xy)^k$, then $\Psi_3(b,y) = \sum \limits_{l=1}^{k-1} (xy)^l \x x(yx)^{k-l-1} + \sum \limits_{l=1}^{k} x(yx)^{l-1} \x (yx)^{k-l}$.
    \item if $b= x(yx)^{k-1}$, then $\Psi_3(b,y) = d\sum \limits_{l=1}^{k-1} (xy)^l \x x(yx)^{k-l-1} + d \sum \limits_{l=1}^{k} x(yx)^{l-1} \x (yx)^{k-l}$.
   \item if $b = y(xy)^i$ for $0 < i \leq k-1$, then $\Psi_3(b,y) = \sum \limits_{l=1}^i \big( (yx)^l \x (yx)^{i-l} + y(xy)^{l-1} \x x(yx)^{i-l} \big) + d\sum \limits_{l=1}^i \big( (yx)^l \x y(xy)^{i-l} + y(xy)^{l-1} \x (xy)^{i-l+1} \big) + dx(yx)^{k-1} \x x(yx)^{i-1} + dy\x (xy)^i + 1 \x (xy)^i  + \delta_{i,1} \big( (xy)^{k-1} \x (yx)^{k-1} + dy(xy)^{k-1} \x (yx)^{k-1} +d^2 y(xy)^{k-1} \x y(xy)^{k-1} \big)$.
   \item $\Psi_3(b, y) = 0$ otherwise.
\end{itemize}
Observe that if $0 \leq i \leq k-1$ then 
\[ \Psi_3 \big( b \x x(yx)^i \x y + x(yx)^i \x y \x b + y \x b \x x(yx)^i \big) = t_2\bigg( x(yx)^i \x t_1 \big( yC(b) \big) +\]
\[+yt_1\Big( b \cdot \sum \limits_{l=0}^{i} (xy)^l \x x \x (yx)^{i-l} + b \cdot \sum \limits_{l=0}^{i-1} x(yx)^l \x y \x x(yx)^{i-l-1} \Big) \bigg),\]
so in order to deal with second and third terms of the sum we only need to know values of obtained formula. Denote this formula by $\Psi_3(2,b,y)$.

\begin{itemize}
    \item If $b = (xy)^j$ for $1 \leq j \leq k$, then
    \begin{enumerate}
        \item  if $i>0$ then $\Psi_3(2,b,y) = \delta_{j,k} \Big( \sum \limits_{l=1}^{i} (xy)^l \x (xy)^{i-l+1} + \sum \limits_{l=1}^{i+1} x(yx)^{l-1} \x y(xy)^{i-l+1} \Big) +  \mu_{i+j, k} \Big(y \x x(yx)^{i+j-k} + d x(yx)^{k-1} \x x(yx)^{i+j-k} + \delta_{i+j, k} \big( d(xy)^{k-1} \x x^2 \big) \Big)$,
        \item if $i=0$ and $j=k$ then $ \Psi_3(2,b,y) = x \x y  + y\x x + dx(yx)^{k-1} \x x  + d(xy)^{k-1} \x x^2$.
    \end{enumerate}
  \item if $b=yx$, then $\Psi_3(2,b,y) = d(xy)^i \x y(xy)^{k-1} +\delta_{i,0} \Big( 1 \x (xy)^{k-1}  + d y \x (xy)^{k-1} + d^2 x(yx)^{k-1} \x (xy)^{k-1}+ d\x x^2  + \sum \limits_{l=1}^{k-1} \big( (xy)^l \x (xy)^{k-l-1}  + x(yx)^{l-1} \x y(xy)^{k-l-1} \big) \Big)  + \mu_{i, 1}\big( dy(xy)^{k-1} \x (yx)^i+  (xy)^{k-1} \x (yx)^i \big)$,
  \item if $b = (yx)^j$ for $j \ge 2$ and $i=0$, then $\Psi_3(2,b,y) =  dy(xy)^{k-1} \x (yx)^{j-1}$,
  \item if $b = y(xy)^j$ for $0 \leq j \leq k-1$, then $\Psi_3(2,b,y) = \mu_{j,1} \delta_{i,0} \big( d y(xy)^{k-1} \x y(xy)^{j-1} +\delta_{j,1}d (yx)^{k-1} \x y^2 \big) +  \mu_{i+j,k-1} \big(d \sum \limits_{l=1}^{k-1} (xy)^l \x x(yx)^{i+j-l} + d \sum \limits_{l=1}^{k} x(yx)^{l-1} \x (yx)^{i+j+1-l} \big) + \mu_{i+j,k} \mu_{i, 1} \big( x(yx)^{k-1} \x x(yx)^{i+j-k} + d \sum \limits_{l=1}^k x(yx)^{l-1} \x (yx)^{i+j+1-l} + d \sum \limits_{l=1}^{k-1} (xy)^{l} \x x(yx)^{i+j-l} + \delta_{i+j,k} (xy)^{k-1} \x x^2 \big)$,
  \item if $b= x(yx)^{k-1}$, then \begin{enumerate}
      \item if $i = 0$ then $\Psi_3(2,b,y) =  x\x 1 + 1 \x x$,
      \item if $i>0$ then $\Psi_3(2,b,y) =  1\x x(yx)^i + \sum \limits_{l=1}^{i} (xy)^{l} \x x(yx)^{i-l} + \sum \limits_{l=1}^{i+1} x(yx)^{l-1} \x (yx)^{i-l+1}$,
  \end{enumerate}
  \item $\Psi_3(2, b, y) = 0$ otherwise.
\end{itemize}
Finally, it remains to compute
$$\Psi_3 \big( (xy)^i \x x \x b + x \x b \x (xy)^i + b \x (xy)^i \x x \big) = t_2 \bigg((xy)^i t_1 \big( x C(b) \big) + xt_1\Big( b C\big( (xy)^i \big) \Big) \bigg)$$
for $1 \leq i \leq k-1$. Denote this by $\Psi_3(3, b, y)$.
\begin{itemize}
    \item If $b = (xy)^j$ for $1 \leq j \leq k$ and $i+j-1 \ge k$, then $\Psi_3(3, b, y) = y(xy)^{k-1} \x y(xy)^{i+j-k-1} + \delta_{i+j,k+1} (yx)^{k-1} \x y^2$,
    \item if $b = y(xy)^j$ for $0 \leq j \leq k-1$ and $i+j \ge k$, then $\Psi_3(3, b, y) = x \x y(xy)^{i+j-k}$,
    \item if $b = x$, then $\Psi_3(3, b, y) = \delta_{i,1} d y(xy)^{k-1} \x y(xy)^{k-2}  + 1 \x y(xy)^{i-1} $,
    \item $\Psi_3(3, b, y) = 0$ otherwise.
\end{itemize}
It remains to note that $q_1w_1 = (xy)^{k-1} = q_2w_2, q_1w_2 = y, q_1w_3 = x + dyx, q_2w_1 = x, q_2w_3 = y$. Now the required formulae follow from direct computations.
\end{proof}

Now we need to understand how $\Delta$ works on elements of degree 4.

\begin{lemma}
$\Delta(x) = 0$ for any monomials $x \in HH^4(R)$.
\end{lemma}
\begin{proof}
Firstly observe that for $a \in HH^4(R)$ following formula holds: $\Delta (a) (1\x 1)= \Delta(a\Psi_4) \Phi_3 (1\x 1) = $
$$\sum \limits_{b \ne 1} \langle (a \Psi_4) (b \cdot x \cdot x \cdot x + x \cdot b \cdot x \cdot x + x \cdot x \cdot b \cdot x + x \cdot x \cdot x \cdot b), 1 \rangle b^* + $$
$$\sum \limits_{b \ne 1} \sum \limits_{i=0}^{k-2} \langle (a \Psi_4) (x \cdot y(xy)^i \cdot x \cdot b + y(xy)^i \cdot x \cdot b \cdot x + x \cdot b \cdot x \cdot y(xy)^i + b \cdot x \cdot y(xy)^i \cdot x), 1 \rangle b^* y(xy)^{k-i-2} +$$ 
$$\sum \limits_{b \ne 1} \sum \limits_{i=1}^{k-1} \langle (a \Psi_4) (x \cdot (yx)^i \cdot y \cdot b + (yx)^i \cdot y \cdot b \cdot x + y \cdot b \cdot x \cdot (yx)^i + b  \cdot x \cdot (yx)^i \cdot y), 1 \rangle b^* (xy)^{k-i-1} +$$ 

$$\sum \limits_{b \ne 1} \langle (a \Psi_4) (b \cdot y \cdot y \cdot y + y \cdot b \cdot y \cdot y + y \cdot y \cdot b \cdot y + y \cdot y \cdot y \cdot b), 1 \rangle b^*(1+dy+d^2 y(xy)^{k-1}) + $$
$$\sum \limits_{b \ne 1} \langle (a \Psi_4) (b \cdot y \cdot x(yx)^{k-1} \cdot y + y \cdot x(yx)^{k-1} \cdot y \cdot b + x(yx)^{k-1} \cdot y \cdot b \cdot y + y \cdot b \cdot y \cdot x(yx)^{k-1}), 1 \rangle b^*(d + dy + d^3 x(yx)^{k-1}) + $$
$$\sum \limits_{b \ne 1} \sum \limits_{i=0}^{k-2} \langle (a \Psi_4) (y \cdot x(yx)^i \cdot y \cdot b + x(yx)^i \cdot y \cdot b \cdot y + y \cdot b \cdot y \cdot x(yx)^i + b \cdot y \cdot x(yx)^i \cdot y), 1 \rangle b^* x(yx)^{k-i-2} +$$ 
$$\sum \limits_{b \ne 1} \sum \limits_{i=1}^{k-1} \langle (a \Psi_4) (y \cdot (xy)^i \cdot x \cdot b + (xy)^i \cdot x \cdot b \cdot y + x \cdot b \cdot y \cdot (xy)^i + b  \cdot y \cdot (xy)^i \cdot x), 1 \rangle b^* (yx)^{k-i-1}.$$ 
Secondly, note that $t_3$ is not equal to zero only on $(xy)^k \x 1$. Denote the evaluation for the $i$-th sum by $\Psi_4(i,b)$ for any $b \in B$. Now for the first sum
\begin{enumerate}
  \item[1.1)] if $b = xy$, then $\Psi_4(1,b) = d \x y(xy)^{k-2}$,
 \item[1.2)] if $b = xyx$, then $\Psi_4(1,b) = d \x (yx)^{k-1} + d \x (xy)^{k-1}$,
 \item[1.3)] if $b = (xy)^k$, then $\Psi_4(1,b) = 1 \x 1 $,
 \item[1.4)] $\Psi_4(1,b) = 0$ otherwise.
\end{enumerate}

For the third sum and $1 \leq i \leq k-1$:
\begin{enumerate}
 \item[3.1)] if $b = yx$, then $\Psi_4(3,b) = d \x y(xy)^{i-1}$,
 \item[3.2)] $\Psi_4(3,b) = 0$ otherwise.
\end{enumerate}

For the fourth sum
\begin{enumerate}
 \item[4.1)] if $b = x(yx)^{k-1}$, then $\Psi_4(4,b) = d \x 1 $,
 \item[4.2)] if $b = (xy)^k$, then $\Psi_4(4,b) = 1 \x 1 $, \
 \item[4.3)] $\Psi_4(4,b) = 0$ otherwise.
\end{enumerate}

For the fifth sum
\begin{enumerate}
 \item[5.1)] if $b = y$, then $\Psi_4(5,b) = d\x 1 $,
 \item[5.2)] $\Psi_4(5,b) = 0$ otherwise.
\end{enumerate}

It is easy to see that other sums gives zero impact, so $\Psi_4(i,b)=0$ for any $b\in B$ and any $i \in \{2, 6, 7\}$.
Now one can deduce that $\Delta(e) = \Delta(p_3e) = 0$, $\Delta(p_2e)= \Delta(x(yx)^{k-1}) = 0$, $\Delta(p_4 e) = 1+1+dy+dy = 0$ and $\Delta(p_1^i e) = \Delta((xy)^i + (yx)^i) = 0$ for any $i \ge 1$, which yields this lemma.
\end{proof}

\subsection{Higher degree elements}
Note that if we know $\Delta(a)$ and $\Delta(b)$, then we don't need to compute $\Delta(ab)$ directly: it is sufficient to know $[a,b]$. Let $a$ be represented by cocycle $f:P_n \longrightarrow A$ and let $b$ be represented by $g:P_m \longrightarrow A$. One can use the following formula:
$$[a,b] = [f\circ \Psi_n, g \circ \Psi_m] \circ \Phi_{n+m-1}.$$

\begin{remark}
Observe that $\Phi_4$ can be visualized directly:
$$\Phi_4 (1 \x 1) = \sum \limits_{b \in B} 1\x b \x x \x x \x x \x b^* + \sum \limits_{b \in B} \sum \limits_{i=0}^{k-2} 1\x b \x x \x y(xy)^i \x x \x y(xy)^{k-i-2}b^* + $$
$$ \sum \limits_{b \in B} \sum \limits_{i=1}^{k-1} 1\x b \x x \x (yx)^i \x y \x (xy)^{k-i-1}b^*  + \sum \limits_{b \in B} 1\x b \x y \x y \x y \x (1+dy+d^2x(yx)^{k-1})b^* +$$
$$  \sum \limits_{b \in B} \sum \limits_{i=0}^{k-2} 1\x b \x y \x x(yx)^i \x y \x x(yx)^{k-i-2}b^* +\sum \limits_{b \in B} \sum \limits_{i=1}^{k-1} 1\x b \x y \x (xy)^i \x x \x (yx)^{k-i-1}b^* +$$
$$ \sum \limits_{b \in B} 1\x b \x y \x x(yx)^{k-1} \x y \x (d+d^2y+d^3x(yx)^{k-1})b^* + d \x x(yx)^{k-1} \x x \x x \x x \x x(yx)^{k-1} +$$
$$ + d \x x(yx)^{k-1} \x x \x y(xy)^{k-2} \x x \x (xy)^{k} + d \x x(yx)^{k-1} \x x \x (yx)^{k-1} \x y \x x(yx)^{k-1} +$$
$$ +d \x x(yx)^{k-1} \x y \x y \x y \x y^2 + d^2 \x x(yx)^{k-1} \x y \x x(yx)^{k-1} \x y \x y^2 +$$
$$ + d\x x(yx)^{k-1} \x y \x (xy)^{k-1} \x x \x x(yx)^{k-1}.$$
\end{remark}

\begin{lemma}
We have $\Delta(q_1e) = \Delta(q_2e) = 0$.
\end{lemma}
\begin{proof}
Fix $a \in HH^1(R)$. By definitions,
$$[a,e] (1 \x 1) = [a \circ \Psi_1, e \circ \Psi_4] \circ \Phi_4 (1\x 1) = \big( (a \Psi_1) \circ (e  \Psi_4) \big)  \Phi_4(1 \x 1) + \big( (e \Psi_4) \circ (a \Psi_1) \big) \Phi_4(1 \x 1).$$
What is $\Psi_4\Phi_4$? It follows from the proofs of above given lemmas that
$$\Psi_3 \Phi_3 (1\x 1) = \Psi_3(x \cdot x \cdot x)  = 1 \x 1,$$
so
$$\Psi_4\Phi_4 (1 \x 1) = \sum \limits_{b \ne 1} t_3 \big( b \Psi_3\Phi_3 (1\x 1) b^* \big) +dt_3 \big( x(yx)^{k-1} \Psi_3 \Phi_3 (1\x 1) x(yx)^{k-1} \big)  = 1.$$ 
Finally, for $a = q_1$ or $a = q_2$ we obtain
$$\big( (a \circ \Psi_1) \circ (e \circ \Psi_4) \big) \circ \Phi_4(1 \x 1) = (a \circ \Psi_1)(1) = 0.$$

Now consider $F^u = (e \circ \Psi_4) \circ (u \circ \Psi_1) = \sum \limits_{i=1}^4 F^u_i$, where $F^u_i = (e \circ \Psi_4) \circ_i (u \circ \Psi_1)$. In order to compute this we need to know $u \Psi_1(b)$ for $u = q_1$ or $q_2$. By direct computations 
$$q_1 \Psi_1 (b) = \begin{cases}
y(xy)^{k-2}, & b =x\\
1+dy, & b=y\\
di x(yx)^i + \delta_{i,1} \cdot y(xy)^{k-1}, & b = x(yx)^i, \ 1 \leq i \leq k-1\\
(xy)^i + (yx)^i + d(i+1)y(xy)^i, & b =y(xy)^i, \ 1 \leq i \leq k-1\\
 di (xy)^i + x(yx)^{i-1}, & b =(xy)^i, \ 1 \leq i \leq k \\
di(yx)^i+x(yx)^{i-1}, & b = (yx)^i, \ 1 \leq i \leq k-1,
\end{cases}$$
$$q_2 \Psi_1 (b) = \begin{cases}
1, & b =x\\
x(yx)^{k-2}+d(xy)^{k-1}, & b=y\\
(xy)^i+(yx)^i, & b = x(yx)^i, \ 1 \leq i \leq k-1\\
x(yx)^{k-1}, & b =yxy\\
y(xy)^{i-1} , & b =(xy)^i, \ 1 \leq i \leq k \\
y(xy)^{i-1} + \delta_{i,1} d x(yx)^{k-1}, & b = (yx)^i, \ 1 \leq i \leq k-1\\
0, & \text{otherwise.}
\end{cases}$$

Since $e\Psi_4 (1 \x b \x a_1 \x a_2 \x a_3 \x 1) = et_3 \Big(b t_2\Big( a_1 t_1 \big( a_2C(a_3) \big) \Big) \Big)$, we only need to calculate $F^u_1$ and $F^u_2$ on elements $b \x a_1 \x a_2 \x a_3$, such that $a_2a_3 \not \in B$. \par 
1) First consider $F_1^u$ for $u \in HH^1(R)$. It is easy to see that $t_2 \big(x t_1(x \x x\x 1) \big) = 1 \x 1 $,  so
$$ et_3\Big( q_i\Psi_1 (b) t_2\big( x t_1(x \x x\x 1) \big) \Big) b^* =  et_3(q_1\Psi_1 (b) \x 1)b^* = 0$$ 
for any $b \in B$ and any $i \in \{1, 2\}$. Hence all summands in $\Phi_4 (1\x 1)$, consisting $x \x x \x x$, give us zero. It remains to check that $t_2 (y t_1 (y \x y \x 1)) = 0$, so
   $$F^{q_1}_1 =  F^{q_2}_1 = 0.$$ \par 
2) For $F_2^u$ (where $u \in HH^1(R)$) one needs to notice that
$$t_3(b t_2(u\Psi(x) t_1(x \x x\x 1))) = t_3(b t_2(u\Psi(x) \x r_x \x 1)) = 0,$$
for any $b \in B$ and $ u \in \{q_1, q_2\}$, so first and eighth sums give us zero. Secondly, $u\Psi_1 t_1 (y \x y \x 1) = u\Psi_1 \x r_y \x 1$, so
$$et_3 \Big(b t_2 \big( u\Psi_1 t_1 (y \x y \x 1) \big) \Big) (1 +dy +d^2x(yx)^{k-1}) = 
\begin{cases}
 dx(yx)^{k-2}, & u = q_2, b = yxy \\
 0, & \text{otherwise. }
\end{cases}$$
Obviously, another sums from $\Phi_4 (1\x 1)$ are coming up with zeros, so
$$F^{q_1}_2 = 0 \text{ and } F^{q_2}_2 =  dx(yx)^{k-2}.$$\par 
3) In order to compute $F^u_3$ notice that $\big( (e \circ \Psi_4) \circ_3 (u \circ \Psi_1) \big) (1 \x b \x a_1 \x a_2 \x a_3 \x 1) = et_3 \Big(b t_2 \big(a_1 t_1 \big( u \Psi_1 (a_2) C(a_3) \big) \big) \Big)$. Then
$$t_2 \big( x t_1 ( q_1 \Psi_1 (x) \x x \x 1 ) \big) = t_2 \big(x t_1 (y(xy)^{k-2} \x x \x 1) \big) = 0,$$
$$t_2 \big( x t_1 ( q_2 \Psi_1 (x) \x x \x 1 ) \big) = t_2 \big( q_2 \Psi_1 (x) t_1 (x \x x \x 1) \big), $$
and all other sums from $\Phi_4 (1\x 1)$ give zero, what can easily be verified via definition of $t_1$ and $t_2$, therefore in this case we can use computations of the previous case. So,
 $$F^{q_1}_3 = F^{q_2}_3 = 0.$$ \par 
 4) Finally, we want to describe $F^u_4$. It is not hard to show that 
 $$\big( (e \circ \Psi_4) \circ_4 (u \circ \Psi_1) \big) (1 \x b \x a_1 \x a_2 \x a_3 \x 1) = et_3 \bigg(b t_2 \Big( a_1 t_1 \big(a_2 C( u \Psi_1 (a_3)) \big) \Big) \bigg).$$ 
 Also
 \[C(u\Psi_1(x)) = \begin{cases}
 \sum \limits_{l=0}^{k-2} (yx)^l \x y \x (xy)^{k-2-l} + \sum \limits_{l=0}^{k-3} y(xy)^l \x x \x y(xy)^{k-3-l} & u =q_1\\
0, & u=q_2
 \end{cases}
\]
and
\[C(u\Psi_1(y)) = \begin{cases}
 d \x y \x 1, & u=q_1\\
\Big( \sum \limits_{l=0}^{k-2} (xy)^l \x x \x (yx)^{k-2-l} + \sum \limits_{l=0}^{k-3} x(yx)^l \x y \x x(yx)^{k-3-l} \Big)(1+dy)+  & \\
 + dx(yx)^{k-2} \x y \x 1, & u=q_2.\\
 \end{cases}
\]
So we have
 $$t_2\bigg(xt_1 \Big(x C\big( q_1 \Psi_1(x) \big) \Big) \bigg) = t_2 \bigg(x t_1 \Big(\sum \limits_{l=0}^{k-2} x(yx)^l \x y \x (xy)^{k-2-l} + \sum \limits_{l=0}^{k-3} (xy)^{l+1} \x x \x y(xy)^{k-3-l} \Big) \bigg) = 0,$$
and
 $$t_2 \bigg(x t_1 \Big(y(xy)^i C \big(q_1\Psi_1 (x) \big) \Big) \bigg) = \delta_{i,0} t_2\big(x \x r_y \x (xy)^{k-2} + dy(xy)^{k-1} \x r_x \x (xy)^{k-2} \big).$$
Further observe that
 $$t_2 \big(x t_1 (q_2 \Psi_1 (x) \x x \x 1) \big) = t_2 \big(q_2 \Psi_1 (x) t_1 (x \x x \x 1) \big),$$
 and by definition of $t_1$ and $t_2$ one can deduce that all other sums from $\Phi_4(1 \x 1)$ give zero impact, so this case also can be reduced to previous ones. So
 $$F^{q_1}_4 = 0 \text{ and } F^{q_2}_4 = dx(yx)^{k-2}.$$
 
Now we only need to compile our previous results and notice that $q_i\Delta(e) = 0$ for $i \in \{1,2\}$:
 $$\Delta(q_1e) (1\x 1 ) = [q_1,e](1\x 1) = \sum \limits_{i=1}^4 F^{q_1}_i = 0,$$
 $$\Delta(q_2e) (1\x 1 ) =[q_2,e](1\x 1) =  \sum \limits_{i=1}^4 F^{q_2}_i = dx(yx)^{k-2} + dx(yx)^{k-2} = 0.$$
\end{proof}

\begin{remark}
For elements of degree six we should understand that is $\Phi_5$:
$$\Phi_5 (1 \x a \x 1) = \sum_b 1 \x a \x b \x x \x x \x x \x b^* + \sum_b \sum_{i=0}^{k-2} 1 \x a \x b \x x \x y(xy)^i \x x \x y(xy)^{k-2-i} b^* $$
$$\sum_b \sum_{i=1}^{k-1} 1 \x a \x b \x x \x (yx)^i \x y \x (xy)^{k-1-i} b^* +  \sum_b 1 \x a \x b \x y \x y \x y \x (1 + dy + d^2 x(yx)^{k-1}) b^* +$$
$$\sum_b \sum_{i=0}^{k-2} 1 \x a \x b \x y \x x(yx)^i \x y \x x(yx)^{k-2-i} b^* + \sum_b \sum_{i=1}^{k-1} 1 \x a \x b \x y \x (xy)^i \x x \x (yx)^{k-1-i} b^* +$$
$$ \sum_b 1 \x a \x b \x y \x x(yx)^{k-1} \x y \x (d + d^2y + d^3x(yx)^{k-1})b^* + d \x a \x x(yx)^{k-1} \x x \x x \x x \x x(yx)^{k-1} +$$
$$+ d \x a \x x(yx)^{k-1} \x x \x y(xy)^{k-2} \x x \x (xy)^k + d \x a \x x(yx)^{k-1} \x x \x (yx)^{k-1} \x y \x x(yx)^{k-1} +$$
$$+ d \x a \x x(yx)^{k-1} \x y \x y \x y \x y^2 + d^2 \x a \x x(yx)^{k-1} \x y \x x(yx)^{k-1} \x y \x y^2 +$$
$$+ d \x a \x x(yx)^{k-1} \x y \x (xy)^{k-1} \x x \x x(yx)^{k-1}.$$
\end{remark}

\begin{lemma}
For all $v \in \{w_1,w_2,w_3\}$ and $e \in HH^4(R)$ all the brackets $[v,e]$ are equal to zero.
\end{lemma}

\begin{proof}
For $v \in HH^2(R)$ and $e \in HH^4(R)$ we have
$$[v,e] (1 \x a \x 1) = \big( (v \Psi_2) \circ (e\Psi_4) \big) \Phi_5 (1 \x a \x 1) + \big( (e\Psi_4)\circ (v \Psi_2) \big) \Phi_5 (1 \x a \x 1).$$
Now we need to show that first summand here is equal to zero. It is obvious that $\big( (v \Psi_2) \circ (e\Psi_4) \big) \Phi_5 = \big( (v \Psi_2) \circ_1 (e\Psi_4) \big) \Phi_5 + \big( (v \Psi_2) \circ_2 (e\Psi_4) \big) \Phi_5$. Let us denote these summands by $S_1^v$ and $S_2^v$ respectively. One can quickly show that $S_2^v$ is equal to zero for all $v$. Indeed,
$$S_2^v (a_1 \x .. \x a_5) = v t_1 (a_1 C(et_3 (a_2 t_2 (a_3 t_1( a_4 C(a_5) ))))),$$
and since $C(1) = 0$ and $et_3(b \x 1) = \begin{cases}1, & b = (xy)^k \\ 0,& \text{ otherwise }
\end{cases}$, the required equality holds.\par 
One can prove that $S_1^v$ equals to zero on all summands from $\Phi_5 (1\x a \x 1)$ except fourth and seventh. Finally,
$$S_1^{v}\big(1 \x a \x (xy)^k \x y \x x(yx)^{k-1} \x y \x (d + d^2 y + d^3x(yx)^{k-1})\big) = \begin{cases} dy + d^2 x(yx)^{k-1}, & a=y, \ v = w_2 \\ dx, &a=y, \ v = w_3 \\ 0, & \text{otherwise,} \end{cases}$$
$$S_1^{v}\big(1 \x a \x (xy)^k \x y \x y \x y \x (1 + d y + d^2 x(yx)^{k-1}) \big) = \begin{cases} dy + d^2 x(yx)^{k-1}, &a=y, \ v = w_2 \\ dx, & a=y, \ v = w_3 \\ 0, & \text{otherwise,} \end{cases}$$
and for other combinations of these sums $S_1^v$ is equal to zero. So $S_1^v (a_1 \x .. \x a_5) = 0$ for any $a_1,..,a_5 \in B$, and hence $\big( (v \Psi_2) \circ (e\Psi_4) \big) \Phi_5 = 0$. \par 
It remains to consider $\big( (e\Psi_4)\circ (v \Psi_2) \big) \Phi_5 = \sum \limits_{i=1}^4 \big( (e\Psi_4)\circ_i (v \Psi_2) \big) \Phi_5$. Denote $\big( (e\Psi_4)\circ_i (v \Psi_2) \big) \Phi_5$ by $F_i^v$. It is easy to see that $F_i^v (a_1 \x .. \x a_5)$ equals to zero for any $i \in \{1, 2, 4\}$ on any summand of $\Phi_5$, except maybe first, fourth, eighth and eleventh summands, because otherwise $t_1 (a_4 C(a_5)) = 0$. \par 
1) Obviously, $t_2 \big( y t_1 (y \x y \x 1) \big) = 0$, so for $F_1^v$ it remains to investigate only first and eighth summands. The computations show that
$$
F_1^v(1\x a \x b \x x \x x \x x \x b^*) = 
$$
$$= \begin{cases}
et_3 \Big(vt_1 \big(x \x x \x y(xy)^{i-1} + y(xy)^{k-1} \x y \x (xy)^{i-1} \big) \x 1 \Big) (xy)^{k-i}, & b=(xy)^i \text{ and } a = x\\
et_3 \Big(v t_1 \big(y \x y \x x(yx)^{i-1} + y^2 \x x \x (yx)^{i-1} \big) \x 1 \Big) (yx)^{k-i}, & b = (yx)^i \text{ and } a = y\\
et_3 \big((yx)^{k} \x 1 \big) y(xy)^{k-2}, & v = w_2, \text{ } b = xyx \text{ and } a = x\\
et_3  \big((xy)^{k} \x 1 \big) x(yx)^{k-2}, & v = w_1, \text{ } b = yxy \text{ and } a=y\\
 et_3 \big((xy)^k \x 1 \big), & v = w_2, \text{ } b = (xy)^k \text{ and } a = y\\
0, & \text{otherwise}
\end{cases}$$
$$= \begin{cases}
1, & v=w_1, \text{ } b=(xy)^k \text{ and } a = x\\
d(xy)^{k-1}, & v=w_3, \text{ } b = xy \text{ and } a = x\\
d(yx)^{k-1}, & v=w_1, \text{ } b = yx \text{ and } a = y\\
y(xy)^{k-2}, & v = w_2, \text{ } b = xyx \text{ and } a = x\\
x(yx)^{k-2}, & v = w_1, \text{ } b = yxy \text{ and } a=y\\
1, & v = w_2, \text{ } b = (xy)^k \text{ and } a = y\\
0, & \text{otherwise.}
\end{cases}$$

So
$$F_1^{w_1} = \begin{cases} 1, & a =x \\ x(yx)^{k-2} + d(yx)^{k-1}, & a = y, \end{cases} \quad F_1^{w_2} = \begin{cases} y(xy)^{k-2}, & a = x \\ 1, & a = y, \end{cases} $$
$$ F_1^{w_3} = \begin{cases} d(xy)^{k-1}, & a = x \\ 0, & a = y. \end{cases}$$
2) For $F_2^v$ we need to examine only summands with numbers one, four, eight and eleven from $\Phi_5$. So for the first summand
$$F_2^{v}(1 \x a \x b \x x \x x \x x \x b^*) = $$
$$
=\begin{cases}
et_3 \big(at_2 (x^3 \x r_x \x 1) \big), &v = w_1 \text{ and } b = (xy)^k\\
et_3\big(at_2(xyx \x r_x \x 1) \big), & v = w_3 \text{ and } b = (xy)^k\\
et_3 \big(at_2((xy)^{k} \ r_x \x 1)\big) y(xy)^{k-2}, & v = w_2 \text{ and } b = xyx \\
et_3\big(ay\x x + dax(yx)^{k-1}\x x + day(xy)^{k-1} \x (yx)^{k-1} \big), & v = w_3 \text{ and } b = y(xy)^{k-1}\\
et_3\big(at_2 (y^2 \x r_x \x 1) \big)(yx)^{k-1}, & v = w_3 \text{ and } b = yx \\
0, & \text{otherwise}
\end{cases}
$$
$$= \begin{cases}
  1, &v = w_1, \text{ } b = (xy)^k \text{ and } a= x \\
  d(yx)^{k-1}, & v = w_3, \text{ } b = (xy)^k \text{ and } a= x \\
y(xy)^{k-2}, & v = w_2, \text{ } b = xyx  \text{ and } a= x\\
 d (yx)^{k-1}, & v = w_3, \text{ } b = y(xy)^{k-1} \text{ and } a= x\\
d (yx)^{k-1}, & v = w_3, \text{ } b = yx \text{ and } a= x\\
0, & \text{otherwise}
\end{cases}$$
and for the forth summand
$$F_2^v \big(1\x a \x b \x y \x y \x y \x (1+dy+d^2x(yx)^{k-1})b^* \big) = $$
$$=\begin{cases} et_3\big( at_2 ((xy)^k \x r_y \x 1)\big)(x(yx)^{k-2}+ d(yx)^{k-1}), & v=w_1, \text{ } a=y \text{ and } b = yxy \\
et_3 \big(at_2((xy)^k \x r_y \x 1) \big) (1+dy+d^2x(yx)^{k-1}), &  v=w_2, \text{ } a=y \text{ and } b = (xy)^k\\
0, & \text{otherwise}
\end{cases} $$
$$= \begin{cases}x(yx)^{k-2}, & v=w_1, \text{ }  a=y \text{ and } b = yxy \\
1, & v=w_2, \text{ } a=y \text{ and } b = (xy)^k\\
0, & \text{otherwise.}
\end{cases}$$
So eighth and eleventh summands give us zero and hence 
$$F_2^{w_1} = \begin{cases}  1, & a =x \\ x(yx)^{k-2}, & a = y, \end{cases} \quad F_2^{w_2} = \begin{cases} y(xy)^{k-2}, & a = x \\ 1, & a = y, \end{cases} \quad F_2^{w_3} = \begin{cases} d(yx)^{k-1}, & a = x \\ 0, & a = y. \end{cases}$$
3) In the case of $F_3^v$ if $t_1 \big(a_3 C(a_4) \big)=0$ then the summand of the form $\sum 1\x a_1 \x .. \x a_5 \x a_6$ give us zero. Also if $(a_3,a_4,a_5) = (x,x,x)$, then $t_1 \big(vt_1(x \x x \x 1) \x x \x 1 \big) = \begin{cases} 1\x r_x \x 1, & v = w_1 \\ 0, & v \not= w_1 \end{cases}$ and 
$$F_3^{w_1} (1 \x x \x (xy)^k \x x \x x \x x \x 1) = et_3 \big((xy)^k \x 1 \big) =  1.$$
There is a zero in all other cases. For the forth sum $F_3^{w_1}=0$, and if $v =w_2$ then
$$
 F_3^{v} (1 \x a \x b \x y \x y \x y \x 1) = 
$$
$$=\begin{cases}
et_3 \big((xy)^k\x 1 \big) (1 + dy), & v = w_2, \text{ } b = (xy)^k \text{ and } a = y \\
et_3 \big(a t_2 ((yx)^k \x r_y \x 1)\big), & v = w_3, \text{ } b = y(xy)^{k-1} \\
et_3 \big(d(xy)^k \x (yx)^{k-1} + d^2 (xy)^k \x y(xy)^{k-1} \big), & v = w_3, \text{ } b = (xy)^k \text{ and } a = x 
\end{cases} 
$$

\noindent so if $b=(xy)^k$ for $v=w_2$ or if $b=y(xy)^{k-1}$ for $v=w_3$ this sum equals to $  (1 + dy) (1+dy+d^2x(yx)^{k-1}) = 1$ for $a=y$. If $b = (xy)^k$ for $v=w_3$ this sum equals to $ d(yx)^{k-1}$ for $a=x$.\\
On the seventh summand $F_3^{v} = 0$, if $v \not=w_3$. And in the case of $v=w_3$ we have
$$F_3^{w_3} \big(1 \x a \x b \x y \x x(yx)^{k-1} \x y \x (d+d^2y +d^3 x(yx)^{k-1}) \big) = $$
$$=\sum_b et_3\Big( at_2 \big( bt_1(xy \x y \x 1 + yx \x y \x 1 + 2dyxy \x y \x 1) \big) \Big)(d+d^2y +d^3 x(yx)^{k-1}) = $$
$$=F_3^{w_3} \big(1 \x a \x b \x y \x y \x y \x (1+dy +d^2 x(yx)^{k-1})\big),$$
so we don't need to compute it, since these two summands kill each other. All other combinations give us zero, so eighth and eleventh sums also give us zero. So \\
$$F_1^{w_1} = \begin{cases} 1, & a =x\\  0, &  a =y, \end{cases} \quad F_1^{w_2} = \begin{cases}  0,& a = x\\  1, & a = y, \end{cases} \quad F_1^{w_3} = \begin{cases} d(yx)^{k-1} + d(yx)^{k-1}, & a = x\\ x +x, & a = y \end{cases} = 0.$$
4) Finally we need to know evaluation of $F_4^v$ in first, forth, eighth and eleventh summands from $\Phi_5$. It is obvious that for the first summand $F_4^v = 0$ if $v \not= w_1$. But
$$F_4^{w_1} (1\x a \x b \x x \x x \x x \x 1)=et_3 (a t_2 (b \x r_x \x 1)) = F_3^{w_1} (1\x a \x b \x x \x x \x x \x 1),$$
so first summand gives $\delta_{a,x} 1$, and $F_4^{w_1}$ is equal to zero for eighth sum. Further 
$$F_4^{w_2} (1\x a \x b \x y \x y \x y \x 1) = et_3 (a t_2 (b \x r_y \x 1)) = F_3^{w_2} (1\x a \x b \x y \x y \x y \x 1),$$
so $F_4^{w_2}$ gives $\delta_{a,y}1$ on the forth sum and $F_4^{v}$ gives zero $v \not= w_2$ for forth and eighth sums. So
$$F_1^{w_1} = \begin{cases} 1, & a =x\\  0, & a =y, \end{cases} \quad F_1^{w_2} = \begin{cases} 0, & a = x\\ 1, & a = y, \end{cases} \quad F_1^{w_3} = 0.$$

It remains to compute $[v,e] =\sum \limits_{i=1}^2 S_i^v + \sum \limits_{i=1}^4 F_i^v = 0$, for any $v \in \{w_1, w_2, w_3\}$, hence the required formulae hold.
\end{proof}
\begin{corollary}
 We have $\Delta(ve)=0$ for any $v \in \{w_1,w_2,w_3\}$.
\end{corollary}
\begin{proof}
$\Delta(v)=0$ for any $v \in \{w_1,w_2,w_3\}$ by Lemma 3, and $\Delta(e)=0$ by Lemma 5. So by Tradler's equation we have
$$\Delta(ve) = \Delta(v) e + v \Delta(e) + [v,e] = 0 \cdot e + v \cdot 0 + 0 = 0$$
for $v \in \{w_1,w_2,w_3\}$.
\end{proof}

\section{Main result}
\begin{theorem}
Let $R=R(k,0,d)$ over an algebraically closed field $K$ of characteristic 2, and let $\Delta$ be $BV$-operator from Theorem 1. Then
\begin{enumerate}
    \item $\Delta$ is equal to 0 on the generators of $HH^*(R)$ from the set $\mathcal{X}$;
    \item $\Delta$ satisfies the equalities 
    \begin{itemize}
        \item in degree 1: $\begin{cases}  \Delta(p_1q_1) = dp_1, \ \Delta(p_1q_2) = \Delta(p_2q_1) = dp_2, \Delta(p_4q_1) = p_2\\
    \Delta(p_3q_1) = \Delta(p_2q_2) = p^{k-1}_1, \ \Delta(p_4q_2) = p_3;
    \end{cases}$
    \item in degree 3: $\begin{cases} \Delta(q_1w_1) = \Delta(q_2w_2) = p_1^{k-2} w_3, \\ \Delta(q_1w_2) = \Delta(q_2w_3) = (1+d^3(l+1)p_4) q_1^2 + dw_2 ,\\
    \Delta(q_2w_1) = q_2^2, \ \Delta(q_1w_3) + \Delta(q_2w_1) = dw_3.
    \end{cases}$
    \end{itemize}
    \item $\Delta(ab) = 0$ on all other combinations of generators $a, b \in \mathcal{X}$.
\end{enumerate}
\end{theorem}
\begin{proof}
The equalities of degree 1 come from Lemma 2 and the equalities of degree 3 come from Lemma 4. Statements 1 and 3 now come from Lemmas 3-6 and Corollary 5 and it is clear from Lemmas above that there are no other conditions for $\Delta$.
\end{proof}

\end{document}